\documentclass[12pt,a4paper]{article}
\pdfoutput=1

\usepackage[ngerman,english]{babel}
\usepackage{a4wide}
\usepackage{graphics}
\usepackage{dsfont}
\usepackage{amsmath}
\usepackage{amssymb}

\usepackage{enumerate, amssymb, amsmath, amsthm}
\newcommand{\N}{\ensuremath{{\mathbb N}}}
\newcommand{\C}{\ensuremath{{\mathbb C}}}
\newcommand{\bm}[1]{{\mbox{{\boldmath$#1$}}}}        
\usepackage{tikz} 
\usepackage[T1]{fontenc}
\usepackage{lmodern}

\newtheorem{satz}{Theorem}[section]
\newtheorem{defi}{Definition} [section]
\newtheorem{lem}[satz]{Lemma} 
\newtheorem{prop}[satz]{Proposition} 
 
\newtheorem{koro}[satz]{Corollary} 
 
\newtheorem{anmerk}[satz]{Remark}

\newcommand{\R}{\ensuremath{{\mathbb R}}}
\DeclareMathOperator*{\Exp}{Exp}

\begin{document}
\begin{center}
{\LARGE{\textbf{Phasetype distributions, autoregressive processes and overshoot}}}
\end{center}
\begin{center}\vspace{.8cm}{\textbf{ S\"oren Christensen}} 

Christian-Albrechts-Universit\"at, Mathematisches Seminar, Kiel, Germany
\end{center}\vspace{.8cm}
\textbf{Abstract:} Autoregressive processes are intensively studied in statistics and other fields of applied stochastics. For many applications the overshoot and the threshold-time are of special interest. When the upward innovations are in the class of phasetype distributions we determine the joint distribution of this two quantities and apply this result to problems of optimal stopping. Using a principle of continuous fit this leads to explicit solutions.\vspace{.8cm}

\vspace{.8cm}


\section{Introduction}\label{sec:setting}
Autoregressive processes play an important role in many areas of applied probability and statistics. They can be seen as one of the building blocks for many models in time-series analysis and estimation and testing techniques are well developed. In this article we study the following setting:\\
Let $0<\lambda\leq 1$, $(Z_n)_{n\in\N}$ be a sequence of independent and identically distributed random variables on a probability space $(\Omega,\mathcal{A},P)$ and let $(\mathcal{F}_n)_{n\in\N}$ be the filtration generated by $(Z_n)_{n\in\N}$. Define the autoregressive process of order 1 (AR(1)-process) $(X_n)_{n\in\N_0}$ by 
\[X_n=\lambda X_{n-1}+Z_n\mbox{~~~for all $n\in\N$}\]
i.e.
\[X_n=\lambda^nX_0+\sum_{k=0}^{n-1}\lambda^kZ_{n-k}.\]
The random variables $(Z_n)_{n\in\N}$ are called the innovations of $(X_n)_{n\in\N_0}$. Using the difference notation the identity $X_n=\lambda X_{n-1}+Z_n$ can be written as 
\[\Delta X_{n}=-(1-\lambda)X_{n-1} \Delta n+\Delta L_{n},\]
where $\Delta X_{n}=X_{n}-X_{n-1}$, $\Delta n=n-(n-1)=1$ and $\Delta L_{n}=\sum_{k=1}^{n}Z_k-\sum_{k=1}^{n-1}Z_k=Z_{n}$.\\
This shows that AR(1)-processes are the discrete-time analogon to (Lévy-driven) Ornstein-Uhlenbeck processes. We just want to mention that many arguments in the following can be carried over to Ornstein-Uhlenbeck processes as well.\\
Autoregressive processes were studied in detail in the last decades. The joint distribution of the threshold-time
\[\tau_b=\inf\{n\in\N_0:X_n\geq b\}\]
and the overshoot 
\[X_{\tau_b}-b\mbox{~~over a fixed level $b$}\]
was of special interest. If $\lambda=1$ the process $(X_n)_{n\in\N_0}$ is a random walk and many results about this distributions are well known. Most of them are based on techniques using the Wiener Hopf-factorization -- see \cite[Chapter VII]{F}  for an overview. Unfortunately no analogon to the Wiener-Hopf factorization is known for AR(1)-processes, so that other ideas are necessary. To get rid of well-studied cases we assume that $\lambda<1$ in the following.\\

This first passage problem for AR(1)-processes was considered in different applications, such as signal detection and surveillance analysis, cf. \cite{FS}. In applications the distribution of this quantities are approximated using Monte-Carlo simulations or Markov chain approximations, cf. e.g. \cite{RP}. But e.g. for questions of optimization analytic solutions are necessary.\\
Using martingale techniques exponential bounds for the expectation of $\tau_b$ can be found. Most of these results are based on martingales defined by using integrals of the form
\begin{equation}
\int_0^\infty e^{uy-\phi(u)}u^{v-1}du\label{eq:int_novikov}
\end{equation}
where $\phi$ is the logarithm of the Laplace transform of the stationary distribution discussed in Section \ref{martingales}. For the integral to be well defined it is necessary that $E(e^{uZ_1})<\infty$ for all $u\in[0,\infty)$ -- cf. \cite{NK} and the references therein.\\
On the other hand if one wants to obtain explicit results about the joint distribution of $\tau_b$ and the overshoot it is useful to assume $Z_1$ to be exponentially distributed. In this case explicit results are given in \cite[Section 3]{CIN} by setting up and solving differential equations. Unfortunately in this case not all exponential moments of $Z_1$ exist and the integral described above cannot be used.\\

The contribution of this article is twofold:
\begin{enumerate}
\item We find the joint distribution of $\tau_b$ and the overshoot for a wide class of innovations: We assume that $Z_1=S_1-T_1$, where $S_1$ and $T_1$ are independent, $S_1$ has a phasetype distribution and $T_1\geq0$ is arbitrary. This generalizes the assumption of exponentially distributed innovations to a much wider class. In Section \ref{phasetype} we establish that $\tau_b$ and the overshoot are -- conditioned on certain events -- independent and we find the distribution of the overshoot. In Section \ref{martingales} we use a series inspired by the integral (\ref{eq:int_novikov}) to construct martingales with the objective of finding the distribution of $\tau_b$.\\ This leads to explicit expressions for expectations of the form $E_x(\rho^{\tau_b}g(X_{\tau_b}))$ for general functions $g$ and $\rho\in(0,1)$.
\item As an application we consider the (Markovian) problem of optimal stopping for $(X_n)_{n\in\N_0}$ with discounted non-negative continuous gain function $g$, i.e. we study the optimization problem
\[v(x)=\sup_{\tau\in\mathcal{T}} E_x(\rho^\tau g(X_\tau))=\sup_{\tau\in\mathcal{T}} E(\rho^\tau g(\lambda^\tau x+X_\tau)),~~x\in\R,~~0<\rho<1,\]
where $\mathcal{T}$ denotes the set of stopping times with respect to $(\mathcal{F}_n)_{n\in\N_0}$; to simplify notation here and in the following we set the payoff equal to $0$ on $\{\tau=\infty\}$. Just very few results are known for this problem. In \cite{N} and \cite{CIN} the innovations are assumed to be exponentially distributed and in \cite{F2} asymptotic results were given for $g(x)=x$.\\
Following the approach described in \cite{CIN} the problem can be reduced to determining an optimal threshold. This is summarized in Section \ref{threshold}. In a second step we use the joint distribution of $\tau_b$ and the overshoot to find the optimal threshold. To this end we use the principle of continuous fit, that is established in Section \ref{cont_fit}. An example is given in Section \ref{example_autoreg}.
\end{enumerate}

\section{Innovations of phasetype}\label{phasetype}
In this section we recall some basic properties of phasetype distributions and identify the connection to AR(1)-processes. In the first subsection we establish the terminology and state some well-known results, that are of interest for our purpose. All results can be found in \cite{As} discussed from the perspective of queueing theory.\\
In the second subsection we concentrate on the threshold-time distribution for autoregressive processes when the positive part of the innovations is of phasetype. The key result for the next sections is that -- conditioned to certain events -- the threshold-time is independent of the overshoot and the overshoot is phasetype distributed as well.

\subsection{Definition and some properties}
Let $m\in\N$, $E=\{1,...,m\}$, $\Delta=m+1$ and $E_\Delta=E\cup\{\Delta\}$.\\
In this subsection we consider a Markov chain $(J_t)_{t\geq0}$ in continuous time with state space $E_\Delta$. The states $1,...,m$ are assumed to be transient and $\Delta$ is absorbing. Denote the generator of $(J_t)_{t\geq0}$ by $\hat{Q}=(q_{ij})_{i,j\in E_\Delta}$, i.e.
\begin{align*}
\hat{q}_{ij}(h):=P(J_{t+h}=j|J_t=i)=q_{ij}h+o(h)\mbox{~for all $i\not=j\in E_\Delta$ and}\\
\hat{q}_{ii}(h):=P(J_{t+h}=i|J_t=i)=1+q_{ii}h+o(h)\mbox{~for all $i\in E_\Delta$ and $h\rightarrow 0, t\geq 0$.}
\end{align*}
If we write $\hat{Q}(h)=(\hat{q}_{ij}(h))_{i,j\in E_\Delta}$ for all $h\geq 0$, then $(\hat{Q}(h))_{h\geq0}$ is a semigroup and the general theory yields that 
\[\hat{Q}(h)=e^{\hat{Q}h}\mbox{~for all $h\geq 0$.}\]
Since $\Delta$ is assumed to be absorbing $\hat{Q}$ has the form 
\[\hat{Q}=\begin{pmatrix} Q & -Q \bm{1} \\ 0...0 & 0 \end{pmatrix}\]
for an $m\times m$-matrix $Q$, where $\bm{1}$ denotes the 	column-vector with entries 1.\\
We consider the survival time of $(J_t)_{t\geq0}$, i.e. the random variable
\[\eta=\inf\{t\geq0:J_t=\Delta\}.\]
Let $\hat{\alpha}=(\alpha,0)$ be an initial distribution of $(J_t)_{t\geq0}$. Here and in the following $\alpha=(\alpha_1,...,\alpha_m)$ is assumed to be a row-vector.

\begin{defi} $P_{\hat{\alpha}}^\eta$ is called a distribution of phasetype with parameters $(Q,\alpha)$ and we write $P_{\hat{\alpha}}^\eta=PH(Q,\alpha)$ for short.
\end{defi}

Let $m=1$ and $Q=(-\beta)$ for a parameter $\beta>0$. In this case it is well-known that $\eta$ is exponentially distributed with parameter $\beta$. This special case will be the key example we often think of. Furthermore let us mention that the class of phasetype distributions is stable under convolutions and mixtures. This shows that the important classes of Erlang- and hyperexponential distributions are of phasetype.\\
Exponential distributions have a very special structure, but phasetype distributions are flexible:

\begin{prop}
The distributions of phasetype are dense in the space of all probability measures on $(0,\infty)$ with respect to convergence in distribution.
\end{prop}

\begin{proof}
See \cite[III, Theorem 4.2]{As}.
\end{proof}

The definition of phasetype distributions does not give rise to an obvious calculus with these distributions, but the theory of semigroups leads to simple formulas for the density and the Laplace-transform as the next lemma shows. All the formulas contain matrix exponentials. The explicit calculation of such exponentials can be complex in higher dimensions, but many algorithms are available for a numerical approximation.

\begin{prop} \label{prop:PH_Vert}
\begin{enumerate}[(a)]
\item The eigenvalues of $Q$ have negative real part.
\item The distribution  function of $PH(Q,\alpha)$ is given by 
\[H_\alpha(s):=P_\alpha(\eta\leq s)=1-\alpha e^{Qs}\bm{1},~~~s\geq0.\]
\item The density is given by
\[h_\alpha(s)=\alpha e^{Qs}q,~~~~s\geq0\]
where $q=-Q\bm{1}$.
\item For all $s\in \C$ with $E_{\hat{\alpha}}(e^{\Re(s) \eta})<\infty$ it holds that
\[\hat{H}_\alpha(s):=E_{\hat{\alpha}}(e^{s\eta})=\alpha(-sI-Q)^{-1}q,\]
where $I$ is the $m\times m$-identity matrix.\\
In particular $\hat{H}_\alpha$ is a rational function.
\end{enumerate}
\end{prop}

\begin{proof}
See \cite[II, Corollary 4.9 and III, Theorem 4.1]{As}.
\end{proof}

An essential property for the applicability of the exponential distribution in modeling and examples is the memoryless property, which even characterizes the exponential distribution. The next lemma can be seen as a generalization of this property to distributions of phasetype.

\begin{lem}\label{lem:gedächtnislosigkeit}
Let $t\geq 0$ and write 
\[H_\alpha^t(s)=P_{\hat{\alpha}}(\eta\leq s+t|\eta\geq t)\mbox{~~for all~$s\geq0$}.\]
Then $H_\alpha^t$ is a 	distribution function of a phasetype distribution with parameters $(Q,\pi^t)$, where $\pi^t_i=P_{\hat{\alpha}}(J_t=i|\eta\geq t)$ for all $i=1,...,m$.
\end{lem}

\begin{proof}
By Proposition \ref{prop:PH_Vert} the random variable $\eta$ has a continuous distribution. Therefore using the Markov-property of $(J_t)_{t\geq0}$ we obtain
\begin{align*}
P_{\hat{\alpha}}(\eta\leq t+s|\eta\geq t)&=P_{\hat{\alpha}}(\eta\leq t+s|\eta>t)\\
&=\sum_{i\in E} P_{\hat{\alpha}}(\eta\leq t+s,J_t=i)\frac{1}{P_{\hat{\alpha}}(\eta>t)}\\
&=\sum_{i\in E} P_{\hat{\alpha}}(\eta\leq t+s|J_t=i)\frac{P_{\hat{\alpha}}(J_t=i)}{P_{\hat{\alpha}}(\eta>t)}\\
&=\sum_{i\in E} P_{e_i}(\eta\leq s)\pi_i^t=\sum_{i\in E} \pi_i^t H_{e_i}(s),
\end{align*}
where $e_i$ is the $i$-th unit vector.
\end{proof}

For the application to autoregressive processes we need the generalization of the previous lemma to the case that the random variable is not necessarily positive.

\begin{lem}\label{lem:gedächtnislos2}
Let $S,T\geq 0$ be stochastically independent random variables, where $S$ is $PH(Q,\alpha)$-distributed. Furthermore let $r\geq0$ and $Z=S-T$. Then 
\[P_{\hat{\alpha}}(r\leq Z\leq r+s)=\sum_{i\in E}\lambda_i(r) H_{e_i}(s),\mbox{~~~$s\geq0$},\]
where $\lambda_i(r)=\int P_{\hat{\alpha}}(J_{r+t}=i) P_{\hat{\alpha}}^T(dt)$. 
\end{lem}

\begin{proof}
Using Lemma \ref{lem:gedächtnislosigkeit} it holds that 
\begin{align*}
P_{\hat{\alpha}}(r\leq Z\leq r+s)&=\int P_{\hat{\alpha}}(r+t\leq S\leq r+s+t)P_{\hat{\alpha}}^T(dt)\\
&=\sum_{i\in E}\int H_{e_i}(s)P_{\hat{\alpha}}(J_{t+r}=i)P_{\hat{\alpha}}^T(dt).
\end{align*}
\end{proof}

\subsection{Phasetype distributions and overshoot of AR(1)-processes}
We again consider the situation of Section \ref{sec:setting}. In addition we assume the innovations to have the following structure:
\[Z_n=S_n-T_n\mbox{~~~~~~for all $n\in\N$,}\]
where $S_n$ and $T_n$ are non-negative and independent and $S_n$ is $PH(Q,\alpha)$-distributed. In this context we remark that each probability measure $Q$ on $\R$ with $Q(\{0\})=0$ can be written as $Q=Q_+*Q_-$ where $Q_+$ and $Q_-$ are probability measures with $Q_+((-\infty,0))=Q_-((0,\infty))=0$ and $*$ denotes convolution (cf. \cite[p.383]{F}).

As a motivation we consider the case of exponentially distributed innovations. If $Z_n$ is exponentially distributed then it holds that for all $\rho\in(0,1]$ and measurable $g:\R\rightarrow[0,\infty)$
\begin{equation}\label{eq:unabh_exponential}
E_x(\rho^\tau g(X_\tau))=E_x(\rho^\tau)E(g(R+b)),
\end{equation}
where $\tau=\tau_b$ is a threshold-time, $x< b$ and $R$ is exponentially distributed with the same parameter as the innovations (cf. \cite[Theorem 3.1]{CIN}). This fact is well known for random walks, cf. \cite[Chapter XII.]{F}. The representation of the joint distribution of overshoot and $\tau$ reduces to finding a explicit expression of the Laplace-transform of $\tau$. In this subsection we prove that a generalization of this phenomenon holds in our more general situation.\\

To this end we use an embedding of $(X_n)_{n\in\N_0}$ into a stochastic process in continuous time as follows:\\
For all $n\in\N$ denote the Markov chain which generates the phasetype-distribution of $S_n$ by $(J^{(n)}_t)_{t\geq0}$ and write 
\[J_t=J^{(n_t+1)}_{t-\sum_{k=1}^{n_t}S_k},\mbox{~where~}n_t=\max\{n\in\N_0:\sum_{k=1}^nS_k\leq t\}\mbox{~~for all $t\geq 0$ }.\]
Hence the process $(J_t)_{t\geq0}$ is constructed by compounding the processes $J^{(n)}$ restricted to their lifetime. Obviously $(J_t)_{t\geq0}$ is a continuous time Markov chain with state space $E$, as one immediately checks, cf. \cite[III, Proposition 5.1]{As}. Furthermore we define a process $(Y_t)_{t\geq0}$ by
\[Y_t=\lambda X_{n_t}-T_{n_t+1}+t-\sum_{k=1}^{n_t} S_{k}.\]
See Figure \ref{fig:Y} for an illustration. It holds that
\[X_n=Y_{(S_1+...+S_n)-}\mbox{~~~for all $n\in\N$,}\]
so that we can find $(X_n)_{n\in\N_0}$ in $(Y_t)_{t\geq0}$. Now let $\hat{\tau}$ be the threshold-time of the process $(Y_t)_{t\geq0}$ over the threshold $b$, i.e.
\[\hat{\tau}=\inf\{t\geq 0: Y_t\geq b\}.\]
By definition of $(Y_t)_{t\geq0}$ it holds that
\begin{equation}\label{eq:char_tau_dach}
Y_t=b~~ \Leftrightarrow ~~t=-\lambda X_{n_t}+T_{n_t+1}+\sum_{k=1}^{n_t}S_k+b\mbox{~for all~} t\geq0.
\end{equation}

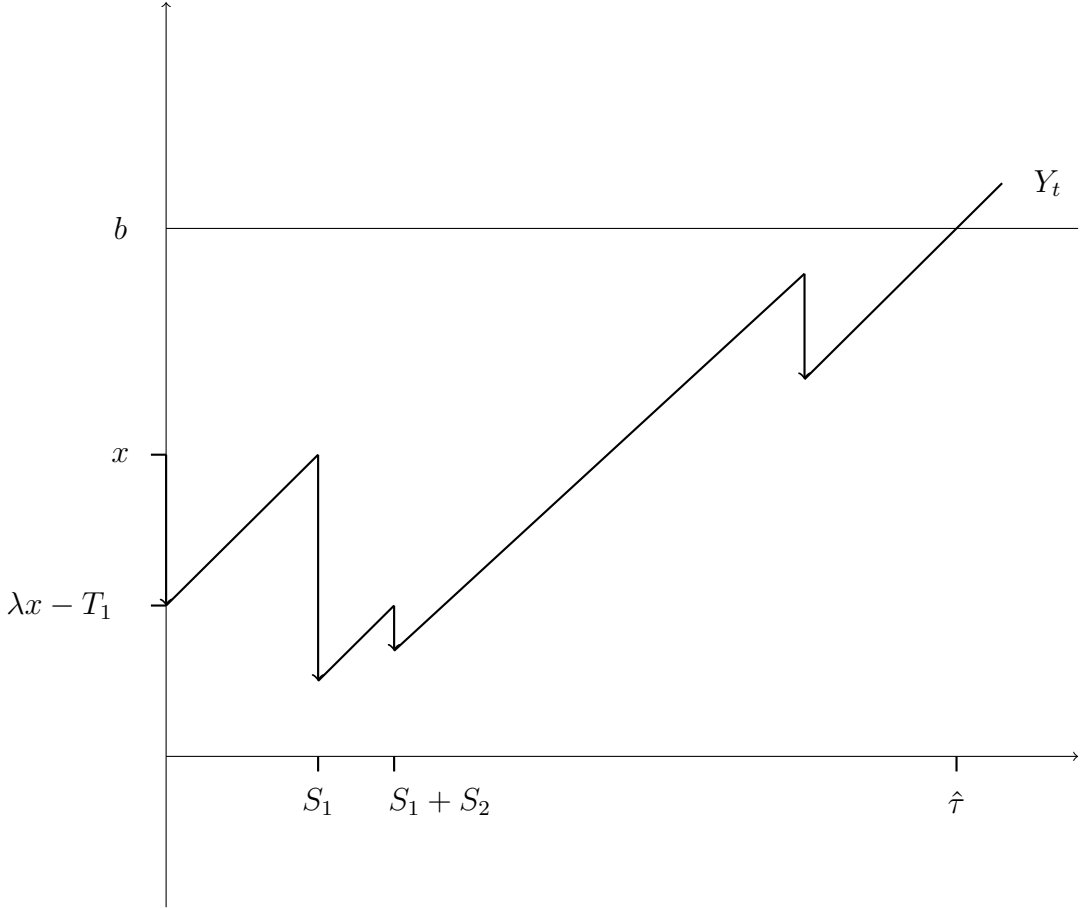
\begin{figure}[ht]
\begin{center}

\begin{tikzpicture}[scale=2]
\draw [->] (0,0) -- (6,0);
\draw [->] (0,-1) -- (0,5);
\draw (0,3.5) -- (6,3.5);
\begin{scope}[thick]
\draw (-0.3,3.5) node {$b$};
\draw (-0.3,2) node {$x$};
\draw (-0.1,2) -- (0,2);
\draw [->] (0,2) -- (0,1);
\draw (-0.7,1) node {$\lambda x-T_1$};
\draw (-0.1,1) -- (0,1);
\draw  (0,1) -- (1,2);
\draw [->](1,2) -- (1,0.5);
\draw (1,-0.3) node {$S_1$};
\draw (1,-0.1) -- (1,0);
\draw (1,0.5) -- (1.5,1);
\draw [->](1.5,1) -- (1.5,0.7);
\draw (1.8,-0.3) node {$S_1+S_2$};
\draw (1.5,-0.1) -- (1.5,0);
\draw (1.5,0.7) -- (4.2,3.2);
\draw [->](4.2,3.2) -- (4.2,2.5);
\draw (4.2,2.5) -- (5.5,3.8);
\draw (5.8,3.8) node {$Y_t$};
\draw (5.2,-0.3) node {$\hat{\tau}$};
\draw (5.2,-0.1) -- (5.2,0);
\end{scope}[thick]

\end{tikzpicture}
\caption{A path of $(Y_t)_{t\geq0}$}\label{fig:Y}
\end{center}
\end{figure}

For the following result we need the event that the associated Markov chain is in state $i$ when $(Y_t)_{t\geq0}$ crosses $b\geq 0$, i.e. the event \[G_i=\{J_{\hat{\tau}}=i\}\mbox{~~~for~}i\in E.\]
For the following considerations we fix the threshold $b\geq 0$. \\
In generalization of the result for exponential distributed innovations the following theorem states that -- conditioned on $G_i$ -- the threshold-time and the overshoot are independent and the overshoot is phasetype distributed as well.

\begin{satz}\label{satz:unabh}
Let $x<b$, $n\in\N$, $y\geq 0$ and write 
\[\tau=\tau_b=\inf\{n\in\N_0:X_n\geq b\}.\]
Then
\[P_x(X_\tau-b\leq y,\tau=n)=\sum_{i\in E}H_{e_i}(y)P_x(\tau=n,G_i).\]
\end{satz}

\begin{proof}
Using Lemma \ref{lem:gedächtnislos2} and the identity (\ref{eq:char_tau_dach}) we obtain
\begin{align*}
P_x(X_\tau-b\leq y,\tau=n)&=E_x(\mathds{1}_{\{\tau\geq n\}}P_x(X_n\geq b,X_n-b\leq y|\mathcal{F}_{n-1}))\\
&=E_x(\mathds{1}_{\{\tau\geq n\}}P_x(b\leq X_n\leq b+y|\mathcal{F}_{n-1}))\\
&=E_x(\mathds{1}_{\{\tau\geq n\}}P_x(b-\lambda X_{n-1}\leq Z_n\leq b+y-\lambda X_{n-1}|\mathcal{F}_{n-1}))\\
&=E_x(\mathds{1}_{\{\tau\geq n\}}\sum_{i\in E} H_{e_i}(y)P_x(J^{(n)}_{b-\lambda X_{n-1}+T_n}=i|\mathcal{F}_{n-1}))\\
&=\sum_{i\in E} H_{e_i}(y) P_x(\tau=n,J_{\hat{\tau}}=i).
\end{align*}
\end{proof}

This immediately implies a generalization of (\ref{eq:unabh_exponential}) to the case of general phasetype distributions:

\begin{koro}\label{koro:unabh}
It holds that
\[E_x(\rho^\tau g(X_\tau))=\sum_{i\in E} E_x(\rho^\tau\mathds{1}_{G_i}) E(g(b+R^{i})),\]
where $R^{i}$ is a $Ph(Q,e_i)$-distributed random variable (under $P$).
\end{koro}

\section{Explicit representations of the joint distribution of threshold time and overshoot}\label{martingales}
Corollary \ref{koro:unabh} reduces the problem of finding expectations of the form $E_x(\rho^\tau g(X_\tau))$ to finding $\Phi_i^b(x)=\Phi_i(x)=E_x(\rho^\tau\mathds{1}_{G_i})$ for $\tau=\tau_b$ and $b>x$. The aim of this section is to construct martingales of the form $(\rho^{n\wedge \tau}h(X_{n\wedge\tau}))_{n\in\N}$ as a tool for the explicit representation of $\Phi_i^b(x)$. To this end some definitions are necessary:\\

We assume the setting of the previous section, i.e. we assume that the innovations can be written in the form 
\[Z_n=S_n-T_n\mbox{~~~~for all $n\in\N$,}\]
where $S_n$ and $T_n$ are non-negative and independent and $S_n$ is $PH(Q,\alpha)$-distributed.\\
Let $\exp(\psi)$ be the Laplace-transform of $Z_1$, i.e. $\psi(u)=\log E(e^{u Z_1})$ for all $u\in\C_+:=\{z\in\C:\Re(z)\geq 0\}$ with real part $\Re(u)$ so small that the expectation exists. Since $E(e^{u Z_1})=E(e^{uS_1})E(e^{-uT_1})$ and $T_1\geq0$ Proposition \ref{prop:PH_Vert} yields the existence of $\psi(u)$ for all $u$ with $\Re(u)$ smaller then the smallest eigenvalue of $-Q$. $\psi$ is analytic on this stripe and -- because of independence -- it holds that
\[\psi(u)=\psi_1(u)+\psi_2(u),\]
where $\exp(\psi_1)$ denotes the Laplace-transform of $S_1$ and $\exp(\psi_2)$ is the Laplace-transform of $-T_1$. $\psi_2$ is analytic on $\C_+$ and $\psi_1$ can be analytically extended to $\C_+\setminus Sp(-Q)$ by Proposition \ref{prop:PH_Vert}. Here $Sp({\cdot})$ denotes the spectrum, i.e. the set of all eigenvalues. Hence $\psi$ can be extended to $\C_+\setminus Sp(-Q)$ as well and this extension is again denoted by $\psi$. Note that this extension can not be interpreted from a probabilistic point of view because  $E(e^{u Z_1})$ does not exist for $u\in\C_+$ with too large real part.\\
To guarantee the convergence of $(X_n)_{n\in\N_0}$ we assume a weak integrability condition -- the well known Vervaat condition
\begin{equation}\label{eq:erwartungsert_endlich}
E(\log(1+|Z_1|))<\infty,
\end{equation}
see \cite[Theorem 2.1]{gm} for a characterization of such conditions in the theory of perpetuities. We do not go into details here, but just want to use the fact that $(X_n)_{n\in\N_0}$ converges to a (finite) random variable $\theta$ in distribution, that fulfills the stochastic fixed point equation
\[P^\theta=P^{\lambda\theta}\ast P^{Z_1},\]
where $\ast$ denotes convolution. Since the AR(1)-process has the representation
\[X_n=\lambda^nX_0+\sum_{k=0}^{n-1}\lambda^kZ_{n-k}\]
and convergence in distribution is equivalent to the pointwise convergence of the Laplace-transforms the Laplace-transform $\exp(\phi)$ of $\theta$ fulfills
\begin{equation}\label{eq:laplace-id2}
\phi(u)=\sum_{k=0}^\infty\psi(\lambda^k u)
\end{equation}
for all $u\in\C_+$ such that the Laplace-transform of $S_1$ exists.\\
The right hand side defines a holomorphic function on $\C_+\setminus\hat{P}$ that is also denoted by $\phi$, where we write $\hat{P}=\bigcup_{n\in\N_0}Sp(-\lambda^{-n} Q)$. For the convergence of the series note that -- as described above -- it converges for all $u\in\C_+$ such that $E(e^{u Z_1})<\infty$. For all other $u\in\C_+$ the series also converges since there exists $k_0$ such that $E(e^{\lambda^ku Z_1})<\infty$ for all $k\geq k_0$.\\
Furthermore the identity 
\begin{equation}\label{eq:laplace-id}
\phi(u)=\phi(\lambda u)+\psi(u)
\end{equation}
holds, whenever $u,\lambda u$ are in the domain of $\phi$. To avoid problems concerning the applicability of (\ref{eq:laplace-id}) we assume that 
\begin{equation}\label{eq:eigenwert}
Sp(\lambda^n Q)\cap Sp(Q)=\emptyset\mbox{~~~for all $n\in \N$.}
\end{equation}
We would like to mention, that the function $\phi$ was used and studied in \cite{NK} as well.\\

The next two lemmas are helpful in the construction of the martingales.

\begin{lem}\label{harm1}
Let $\delta\in\C_+$ such that $E(e^{\delta S_1})$ exists. Then for all $x< b$ it holds that 
\[\rho E(e^{\delta(\lambda x+Z_1)}\mathds{1}_{\{\lambda x+Z_1\geq b\}})=\alpha_\delta e^{-\lambda xQ}q,\]
where $\alpha_\delta=\rho\alpha (-\delta I-Q)^{-1}e^{(\delta I+Q)b+\psi_2(-Q)}$.
\end{lem}

\begin{proof}
In the following calculation we use the fact that all matrices commutate and that all eigenvalues of $\delta I+Q$ have negative real part. It holds
\begin{align*}
E(e^{\delta(\lambda x+Z_1)}\mathds{1}_{\{\lambda x+S_1\geq b\}})&=e^{\delta\lambda x}\int_0^\infty E(e^{\delta(S-t)}\mathds{1}_{\{\lambda x+S_1-t\geq b\}})P^T(dt)\\
&=e^{\delta\lambda x}\int_0^\infty\int_{b+t-\lambda x}^\infty e^{\delta(s-t)}\alpha e^{Qs}qds P^T(dt)\\
&=e^{\delta\lambda x}\int_0^\infty e^{-\delta t}\alpha \int_{b+t-\lambda x}^\infty e^{(\delta I+Q)s}dsqP^T(dt)\\
&=-e^{\delta\lambda x}\int_0^\infty e^{-\delta t}\alpha (\delta I+Q)^{-1}e^{(\delta I+Q)(b+t-\lambda x)}qP^T(dt).\\
\end{align*}
We obtain
\begin{align*}
\rho E(e^{\delta(\lambda x+Z_1)}\mathds{1}_{\{\lambda x+Z_1\geq b\}})&=-\rho e^{\delta\lambda x}\alpha \int_0^\infty e^{Q t} P^T(dt) (\delta I+Q)^{-1}e^{(\delta I+Q)(b-\lambda x)}q\\
&=-\rho\alpha\int_0^\infty e^{-Q s} P^{-T}(ds)(\delta I+Q)^{-1}e^{-Q\lambda x} e^{(\delta I+Q)b}q\\
&=\alpha_\delta e^{-Q\lambda x}q.
\end{align*}
\end{proof}

We write $Q_\gamma=-\gamma Q $ for short. For all $\gamma\in\C_+$ fulfilling 
\begin{equation}\label{eq:gamma}
Sp(\lambda^nQ_\gamma)\cap \hat{P}=\emptyset\mbox{~~~for all $n\in\N$}
\end{equation}
we define the function
\[f_\gamma:\R\rightarrow\C^{m\times m}, x\mapsto \sum_{n\in\N}e^{x\lambda^n Q_\gamma-\phi(\lambda^nQ_\gamma)}\rho^{n-1}.\]
This series converges because $e^{x\lambda^n Q_\gamma-\phi(\lambda^nQ_\gamma)}$ is bounded in $n$. Note that the summand of this series is similar to the integrand in (\ref{eq:int_novikov}).

\begin{lem}\label{harm2}
There exists $\delta>0$ such that for all $x\in\R$ and $\gamma\in \C_+$ with $|\gamma|<\delta$ it holds that
\[\rho E({f}_\gamma(\lambda x+Z_1))={f}_\gamma(x)-e^{\lambda x Q_\gamma -\phi(\lambda Q_\gamma)}.\]
\end{lem}

\begin{proof}
For all $\gamma\in\C_+$ with $|\gamma|$ sufficiently small the expected value $E(e^{Q_\gamma\lambda^n Z_1})$ exists for all $n\in\N$ since $Q$ has (finitely many) negative eigenvalues. This leads to
\begin{align*}
E({f}_\gamma(\lambda x+Z_1))&=E(\sum_{n\in\N}e^{(\lambda x+Z_1)\lambda^n Q_\gamma-\phi(\lambda^n Q_\gamma)}\rho^{n-1})\\
&=\sum_{n\in\N}e^{(\lambda x)\lambda^n Q_\gamma-\phi(\lambda^n Q_\gamma)}E(e^{\lambda^n Q_\gamma Z_1})\rho^{n-1}\\
&=\sum_{n\in\N}e^{x\lambda^{n+1} Q_\gamma-\phi(\lambda^n Q_\gamma)+\psi(\lambda^n Q_\gamma)}\rho^{n-1}\\
&\stackrel{(\ref{eq:laplace-id})}{=}\frac{1}{\rho} \sum_{n\in\N}e^{x\lambda^{n+1} Q_\gamma-\phi(\lambda^{n+1} Q_\gamma)}\rho^{n}\\
&=\frac{1}{\rho} ({f}_\gamma(x)-e^{\lambda Q_\gamma x-\phi(\lambda Q_\gamma)}).
\end{align*}
\end{proof}

The next step is to find a family of equations characterizing
\[\Phi(x)=(\Phi_1(x),...,\Phi_m(x))=(E_x(\rho^\tau\mathds{1}_{G_1}),...,E_x(\rho^\tau\mathds{1}_{G_m}))\]
using martingale techniques where $\tau=\tau_b$, $x<b$. To this end we consider 
\[h_{\gamma,\delta}:\R\rightarrow \C, x\mapsto e^{\delta x}\mathds{1}_{\{x\geq b\}}+\beta_{\gamma,\delta} f_{\gamma}(x)q,\]
for all $\delta\in\C_+\setminus SP(-Q)$ and $\gamma$ fulfilling (\ref{eq:gamma}) where $\beta_{\gamma,\delta}=\alpha_\delta {e^{\phi(\lambda Q_\gamma)}}$. For the special value $\gamma=1$ we write $h_\delta=h_{\gamma,\delta}$ and this function is well-defined by (\ref{eq:eigenwert}).\\

Putting together the results of Lemma \ref{harm1} and Lemma \ref{harm2} we obtain the equation
\begin{equation}\label{harm3}
\rho E_x(h_{\gamma,\delta}(X_1))-h_{\gamma,\delta}(x)=\alpha_\delta e^{-\lambda xQ}q-\alpha_\delta e^{-\lambda \gamma xQ}q \mbox{~~~for all $x< b$}
\end{equation}
for all $\gamma, \delta\in\C_+$ with sufficiently small modulus.\\

Before stating the equations we need one more technical result.

\begin{lem}\label{lem:eta}
Let $R^i$ be a $PH(Q,e_i)$-distributed random variable and denote by $\psi^i(\cdot)=\log(e_i(-\cdot I-Q)^{-1}q)$ the holomorphic extension of the logarithmized Laplace-transform of $R_i$. Here $e_i$ denotes the $i$-th unit vector. Let $|\gamma|,|\delta|$ be so small that $E(h_{\gamma,\delta}(b+R^i))$ exists. Then it holds that
\[E(h_{\gamma,\delta}(b+R^i))=e^{\delta b} \alpha_{\gamma,i} (-\delta I-Q)^{-1}q=:\eta_{\gamma,\delta,i} ,\]
where
\[\alpha_{\gamma,i}=e_i+\alpha  e^{Qb+\psi_2(-Q)+\phi(\lambda Q_\gamma)}\sum_{n\in\N}e^{b\lambda^nQ_\gamma-\phi(\lambda^nQ_\gamma)+\psi^i(\lambda^nQ_\gamma)}\rho^{n}.\]
\end{lem}

\begin{proof}
Simple calculus similar to the previous ones yields the result.
\end{proof}

\begin{satz}\label{gleichung}
For all $x<b$ and $\delta\in\C_+\setminus Sp(-Q)$ it holds that
\[\sum_{i=1}^m \eta_{\delta,i} \Phi_i(x)=h_\delta(x),\]
where $\eta_{\delta,i}=\eta_{1,\delta,i}$ is given in the previous Lemma.
\end{satz}

\begin{proof}
Write $h:=h_{\gamma,\delta}$ for $\delta, \gamma\in\C_+$ with $|\delta|,|\gamma|$ so small that $E(h(Z_1))<\infty$.\\
The discrete version of Itô's formula yields
\begin{align*}
&\rho^nh(X_n)-\sum_{i=0}^{n-1}\rho^i(\rho E_x(h(X_{i+1})|X_i)-h(X_i))\\
=&h(X_0)+\sum_{i=0}^{n-1}\rho^{i+1}(h(X_{i+1})-E_x(h(X_{i+1})|X_i))=:M_n
\end{align*}
and $(M_n)_{n\in\N}$ is a martingale. The optional sampling theorem applied to $\tau=\tau_b$ yields
\begin{align*}
h(x)&=E_x(M_{\tau\wedge n})=E_x(\rho^{n\wedge\tau}h(X_{n\wedge\tau}))-E_x\left(\sum_{i=0}^{{n\wedge\tau}-1}\rho^i(\rho E_x(h(X_{i+1})|X_i)-h(X_i))\right)\\
&=E_x(\rho^{n\wedge\tau}h(X_{n\wedge\tau}))-E_x\left(\sum_{i=0}^{{n\wedge\tau}-1}\rho^i(\alpha_\delta e^{-\lambda X_i Q}q-\alpha_\delta e^{-\lambda \gamma X_i Q}q )\right)
\end{align*}
using equality (\ref{harm3}). The dominated convergence theorem shows that
\[h_{\delta,\gamma}(x)=E_x(\rho^{\tau}h_{\delta,\gamma}(X_{\tau}))-E_x\left(\sum_{i=0}^{{\tau}-1}\rho^i(\alpha_\delta e^{-\lambda X_i Q}q-\alpha_\delta e^{-\lambda \gamma X_i Q}q )\right);\]
note that the dominated convergence theorem is applicable to both summands since $Q$ has negative eigenvalues and so $ e^{-sQ}$ is bounded in $s$ for $s$ with $\Re(s)$ being bounded above. Corollary \ref{koro:unabh} leads to
\[h_{\gamma,\delta}(x)=\sum_{i=1}^m E_x(\rho^\tau\mathds{1}_{G_i}) E(h_{\gamma,\delta}(R_i+b))-E_x\left(\sum_{i=0}^{{\tau}-1}\rho^i(\alpha_\delta e^{-\lambda X_i Q}q-\alpha_\delta e^{-\lambda \gamma X_i Q}q )\right),\]
where $R_i$ is $PH(Q,e_i)-$distributed and the previous Lemma implies
\[E(h_{\gamma,\delta}(R_i+b))=e^{\delta b} \alpha_{\gamma,i} (-\delta I-Q)^{-1}q.\]
Since $\sum_{i=0}^{{\tau}-1}\rho^i(\alpha_\delta e^{-\lambda X_i Q}q-\alpha_\delta e^{-\lambda \gamma X_i Q}q )$ is bounded both sides of the equation
\[h_{\delta,\gamma}(x)=\sum_{i=1}^m \eta_{\gamma,\delta,i} \Phi_i(x) -E_x\left(\sum_{i=0}^{{\tau}-1}\rho^i(\alpha_\delta e^{-\lambda X_i Q}q-\alpha_\delta e^{-\lambda \gamma X_i Q}q )\right)\]
are holomorphic in $\{\gamma\in\C_+:\hat{P}\cap Sp(-\gamma\lambda^n Q)=\emptyset\mbox{~for all $n\in\N$}\}$ and the identity theorem for holomorphic functions yields that these extensions agree on their domains. Keeping (\ref{eq:eigenwert}) in mind we especially obtain for $\gamma=1$
\[h_\delta(x)=\sum_{i=1}^m \eta_{\delta,i} \Phi_i(x).\]
Furthermore both sides of the equations are again holomorphic functions in $\delta$ on\linebreak $\C_+\setminus Sp(-Q)$. Another application of the identity theorem proves the assertion.
\end{proof}

The equation in the theorem above appears useful and flexible enough for the explicit solution as shown in the next subsections.

\subsection{The case of exponential positive innovations}
As described above the case of $\Exp(\mu)$-distributed positive innovations is of special interest. In this case we obtain the solution directly from the results above. Hence let $m=1$, $\alpha=1$, $Q=-\mu$ and $q=\mu$. It is not relevant which $\delta$ we take; because the expressions simplify a bit we choose $\delta=0$. Then we obtain
\begin{align*}
\eta_{0,1}&=1+e^{\psi_2(\mu)-\mu b+\phi(\lambda \mu)}\sum_{n\in\N}e^{\lambda^n\mu b-\phi(\lambda^n \mu)+\psi_1(\lambda^n \mu)}\rho^{n}\\
&\stackrel{(\ref{eq:laplace-id})}{=}1+e^{\psi_2(\mu)-\mu b+\phi(\lambda \mu)}\sum_{n\in\N}e^{\lambda^n\mu b-\phi(\lambda^{n+1} \mu)-\psi_2(\lambda^{n} \mu)}\rho^{n}\\
&=e^{\psi_2(\mu)-\mu b+\phi(\lambda \mu)}\sum_{n\in\N_0}e^{\lambda^n\mu b-\phi(\lambda^{n+1} \mu)-\psi_2(\lambda^{n} \mu)}\rho^{n}
\end{align*}
and
\[h_0(x)=e^{\psi_2(\mu)-\mu b+\phi(\lambda \mu)}\sum_{n\in\N}e^{\lambda^n\mu x-\phi(\lambda^n \mu)}\rho^{n}\]
for $x<b$. Theorem \ref{gleichung} yields

\begin{satz}
\begin{align*}
E_x(\rho^\tau)=\frac{h_0(x)}{\eta_{0,1}}=\frac{\sum_{n\in\N}e^{\lambda^n\mu x-\phi(\lambda^n \mu)}\rho^n}{\sum_{n\in\N_0}e^{\lambda^n\mu b-\phi(\lambda^{n+1} \mu)-\psi_2(\lambda^n \mu)}\rho^n}\mbox{~~~for all $x<b$.}
\end{align*}
\end{satz}

In \cite{CIN} the special case of positive exponential distributed innovations was treated by finding and solving ordinary differential equations for $E_x(\rho^\tau)$. For this case -- i.e. $T_1=0$ -- we obtain
\[E_x(\rho^\tau)=\frac{\sum_{n\in\N}e^{\lambda^n\mu x-\phi(\lambda^n \mu)}\rho^n}{\sum_{n\in\N_0}e^{\lambda^n\mu b-\phi(\lambda^{n+1} \mu)}\rho^n}.\]
To get more explicit results we need a simple expression for $\phi(\lambda^n \mu)$. Using identity (\ref{eq:laplace-id2}) we find such an expression as \begin{align*}
e^{\phi(\lambda^n\mu)}&=\prod_{k=0}^\infty e^{\psi_1(\lambda^{n+k}\mu)}=\prod_{k=0}^\infty\frac{\mu}{\mu-\lambda^{n+k}\mu}=\prod_{k=0}^\infty\frac{1}{1-\lambda^{n+k}}\\
&=\frac{\prod_{k=1}^{n-1}(1-\lambda^k)}{\prod_{k=1}^{\infty}(1-\lambda^k)}=\frac{(\lambda,\lambda)_{n-1}}{\phi_e(\lambda)},
\end{align*}
where $(a,q)_n=\prod_{k=1}^{n-1}(1-aq^{k-1})$ denotes the $q$-Pochhammer-symbol and $\phi_e(q)=(q,q)_\infty$ denotes the Euler function. This leads to \[E_x(\rho^\tau)=\frac{\sum_{n\in\N}\frac{\rho^n}{(\lambda,\lambda)_{n-1}}e^{\lambda^n\mu x}}{\sum_{n\in\N_0}\frac{\rho^n}{(\lambda,\lambda)_{n}}e^{\lambda^n\mu b}}\]
and the numerator is given by
\begin{align*}
\sum_{n\in\N}\frac{\rho^n}{(\lambda,\lambda)_{n-1}}e^{\lambda^n\mu x}=&\sum_{k\in\N_0}\frac{(\mu x)^k}{k!}\sum_{n\in\N}\frac{(\rho \lambda^k)^n}{(\lambda,\lambda)_{n-1}}\\
=&\rho \sum_{k\in\N_0}\frac{(\mu x \lambda)^k}{k!}\sum_{n\in\N}\frac{(\rho \lambda^k)^{n-1}}{(\lambda,\lambda)_{n-1}}\\
=&\rho \sum_{k\in\N_0}\frac{(\mu x \lambda)^k}{k!}\frac{1}{(\rho \lambda^k,\lambda)_\infty}\\
=&\frac{\rho}{(\rho,\lambda)_\infty} \sum_{k\in\N_0}\frac{(\rho,\lambda)_k(\mu \lambda x)^k}{k!}.
\end{align*}
Note that we used the $q$-binomial-theorem in the third step (see \cite[(1.3.15)]{gr} for a proof).
An analogous calculation for the denominator yields
\[\sum_{n\in\N_0}\frac{\rho^n}{(\lambda,\lambda)_{n}}e^{\lambda^n\mu b}=\frac{1}{(\rho,\lambda)_\infty}\sum_{k\in\N_0}\frac{(\rho,\lambda)_k (\mu b)^k}{k!}\]
and we obtain

\begin{satz}\label{satz:exp}
If $S_1$ is $\Exp(\mu)$-distributed and $T_1=0$ it holds that
\[E_x(\rho^{\tau_b})=\rho\cdot\frac{\sum_{k\in\N_0}(\rho,\lambda)_k \frac{(\mu x \lambda)^k}{k!}}{\sum_{k\in\N_0}(\rho,\lambda)_k\frac{(\mu b)^k}{k!}}\mbox{~~~~for all $x< b$}.\]
\end{satz}

For the special case $\mu=1$ this formula was obtained by an approach via differential equations based on the generator in \cite[Theorem 3]{N}.\\
Noting that 
\[E_{b}(\rho^{\tau_{b+}})=\rho\cdot\frac{\sum_{k\in\N_0}(\rho,\lambda)_k \frac{(\mu b \lambda)^k}{k!}}{\sum_{k\in\N_0}(\rho,\lambda)_k\frac{(\mu b)^k}{k!}}\]
by direct calculation we find
\[\frac{d}{db} E_x(\rho^{\tau_b})=E_x(\rho^{\tau_b})\mu (E_{b}(\rho^{\tau_{b+}})-1).\]
This reproduces Theorem 3.3 in \cite{CIN}. Note that in that article the stopping time $\tilde{\tau}_b=\inf\{n\in\N_0:X_t>b\}$ was considered. But this leads to analogous results since
\[\tau_{b+}=\tilde{\tau}_b\mbox{~~~under $P_b$}\]
and
\[\tau_b=\tilde{\tau}_b~~~\mbox{$P_x$ a.s. for all~}x\not=b.\]

\subsection{The general case}
Theorem \ref{gleichung} gives a powerful tool for the explicit calculation of $\Phi$ in many cases of interest as follows:\\
By Lemma \ref{lem:eta} we see that $\eta_{\delta,i}$ is a rational function of $\delta$ with poles in $Sp(-Q)$ for all $i=1,..,m$. We assume for simplicity that all eigenvalues are pairwise different (for the general case see the remark at the end of the section). Then partial fraction decomposition yields the representation
\[\eta_{\delta,i}=\sum_{j=1}^m\frac{a_{i,j}}{\mu_j-\delta}\mbox{~~~~for some $a_{i,1},...,a_{i,m}$}\]
and since $h_\delta(x)$ is rational in $\delta$ with the same poles we may write
\[h_\delta(x)=\sum_{j=1}^m\frac{c_j(x)}{\mu_j-\delta}\mbox{~~for some $c_{1}(x),...,c_{m}(x)$}.\]
Theorem \ref{gleichung} reads
\[\sum_{j=1}^m\frac{\sum_{i=1}^ma_{ij}\Phi_i(x)}{\mu_j-\delta}=\sum_{j=1}^m\frac{c_j(x)}{\mu_j-\delta}\]
and the uniqueness of the partial fraction decomposition yields
\[\sum_{i=1}^ma_{ij}\Phi_i(x)=c_j(x),\]
i.e.
\[A\Phi(x)=c(x),\]
where $A=(a_{ij})_{i,j=1}^m$, $c(x)=(c_j(x))_{j=1}^m$. This leads to

\begin{satz}
If $A$ is invertible, then $\Phi(x)$ is given by
\[\Phi(x)=A^{-1}c(x)\mbox{~~for all $x< b$}.\]
\end{satz}

\begin{anmerk}
Note that the assumption of distinct eigenvalues was made for simplicity only. When it is not fulfilled we can use the general partial fraction decomposition formula and obtain the analogous result.
\end{anmerk}

\section{Applications to optimal stopping}\label{threshold}
To tackle the optimal stopping problem 
\[v(x)=\sup_{\tau\in\mathcal{T}} E_x(\rho^\tau g(X_\tau))\]
it is useful to reduce the (infinite dimensional) set of stopping times to a finite dimensional subclass. A often used class for this kind of problems is the class of threshold-times. This reduction can be carried out in two different ways:

\begin{enumerate}[(a)]
\item \label{it:a} We use elementary arguments to reduce the set of potential optimal stopping times to the subclass of threshold-times, i.e. to stopping times of the form
\[\tau_b=\inf\{t\geq 0: X_t\geq b\}\]
for some $b\in\R$. Then we find the optimal threshold. A summary of examples where is approach can be applied is given below.

\item \label{it:b} We make the ansatz that the optimal stopping time is of threshold type, identify the optimal threshold and use a verification theorem to prove that this stopping time is indeed optimal.
\end{enumerate}


On (\ref{it:a}): In \cite[Section 2]{CIN} the idea of an elementary reduction of optimal stopping problems to threshold-problems was studied in detail. For arbitrary innovations this approach can be applied to power gain function, i.e. $g(x)=x^n,n\in\N$. This problem that is known as the Novikov-Shiryaev problem and was completely solved for random walks in \cite{NS} and \cite{NS2}. Furthermore the approach applies to gain functions of call-type $g(x)=(x-K)^+$.\\
For the special case of nonnegative innovations a much wider class of gain functions can be handled such as exponential functions.\\

On (\ref{it:b}): To use the second approach described above the following easy verification theorem is useful.
\begin{lem}
Let $b^*\in\R$, write $v^*(x)=E_x(\rho^{\tau_{b^*}} g(X_{\tau_{b^*}}))$ and assume that
\begin{enumerate}[(a)]
\item $v^*(x)\geq g(x)$ for all $x<b^*$.
\item $E(\rho v^*(\lambda x+Z_1))\leq v^*(x)$ for all $x\in\R$. 
\end{enumerate}
Then $v=v^*$ and $\tau_{b^*}$ is optimal. 
\end{lem}

\begin{proof}
By the independence of $(Z_n)_{n\in\N}$ property $(b)$ implies that $(\rho^nv^*(X_n))_{n\in\N}$ is a supermartingale under each measure $P_x$. Since it is positive the optional sampling theorem leads to
\[v^*(x)\geq \sup_{\tau\in\mathcal{T}}E_x(\rho^\tau v^*(X_\tau))\geq \sup_{\tau\in\mathcal{T}}E_x(\rho^\tau g(X_\tau))\mbox{~~~for all~}x\in\R,\]
where the second inequality holds by $(a)$ since $v^*(x)=g(x)$ for all $x\geq b$. On the other hand $v^*(x)\leq v(x)$, i.e. $v^*(x)=v(x)$ and $\tau_{b^*}$ is optimal.
\end{proof}

\section{On the explicit solution of the optimal stopping problem}\label{explicit}
Now we are prepared to solve the optimal optimal stopping problem
\[v(x)=\sup_\tau E_x(\rho^\tau g(X_\tau)),~~x\in\R.\]
Section \ref{threshold} gives conditions for the optimality of threshold-times. In this cases we can simplify the problem to
\[v(x)=\sup_b E_x(\rho^{\tau_b} g(X_{\tau_b})),~~x\in\R.\]
Now take an arbitrary starting point $x$. Then we have to maximize the real function
\[\Psi_x:(x,\infty) \rightarrow \R, b\mapsto E_x(\rho^{\tau_b} g(X_{\tau_b}))=\sum_{i=1}^m\Phi_i^b(x)E(g(b+R^i)),\]
where $R^i$ is $Ph(e_i,Q)$-distributed, $b\geq0$. The results of the previous section give rise to an explicit calculation of $\Phi_i^b(x)$ and of $\Psi_x$.\\
Hence we are faced with the well-studied maximization problem for real functions, that can -- e.g. -- be solved using the standard tools from differential calculus.\\
If we have found a maximum point $b^*$ of $\Psi_x$ and $\Psi_x(b^*)>g(x)$, then 
\[\tau^*=\inf\{n\in\N_0:X_n\geq b^*\}\]
is an optimal stopping time when $(X_n)_{n\in\N_0}$ is started in $x$. \\
A more elegant approach for finding the optimal threshold $b^*$ is the principle of continuous fit:

\section{The principle of continuous fit}\label{cont_fit}
The principles of smooth and continuous fit play an important role in the study of many optimal stopping problems. The principle of smooth fit was already introduced in \cite{mi} and has been applied in a variety of problems, ranging from sequential analysis to mathematical finance. The principle of continuous pasting is more recent and was introduced in \cite{ps2} as a variational principle to solve sequential testing and disorder problems for the Poisson process. For a discussion in the case of Lévy processes and further references we refer to \cite{CI}. Another overview is given in \cite[Chapter IV.9]{ps} and one may summarize, see the above reference, p. 49:

``If $X$ enters the interior of the stopping region $S$ immediately after starting on $\partial S$, then the optimal stopping point $x^*$ is selected so that the value function $v$ is smooth in $x^*$. If $X$ does not enter the interior of the stopping region immediately, then $x^*$ is selected so that $v$ is continuous in $x^*$.''

Most applications of this principle involve processes in continuous time. In discrete time an immediate entrance is of course not possible, so that one can not expect the smooth-fit principle to hold. In this section we prove that the continuous-fit principle holds in our setting and illustrate how it can be used for an easy determination of the optimal threshold. \\
We keep the notations and assumptions of the previous sections and -- as before -- we assume that the optimal stopping set is an interval of the form $[b^*,\infty)$ and consider the optimal stopping time $\tau_{b^*}=\tau=\inf\{n\in\N_0:X_n\geq b^*\}$.\\
Furthermore we assume that 
\begin{equation}\label{stet}
\lim_{\epsilon\searrow0}\Phi^{b^*}_i(b^*-\epsilon)=\lim_{\epsilon\searrow0}\Phi^{b^*+\epsilon}_i(b^*)\mbox{~~for all~}i=1,...,m.
\end{equation}
Note that this condition is obviously fulfilled in the cases discussed above. If $g$ is continuous under an appropriate integrability condition it furthermore holds that
\begin{equation}\label{g_stet}
E(g(R_i+\epsilon+b^*))\rightarrow E(g(R_i+b^*))\mbox{~~~as~}\epsilon\rightarrow 0,~~i=1,...,m.
\end{equation}

\begin{prop}
Assume (\ref{stet}) and (\ref{g_stet}). Then it holds that
\[\lim_{b\nearrow b^*}v(b)=g(b^*).\]
\end{prop}

\begin{proof}
Let $\epsilon>0$. First note that $v(b)> g(b)$ for all $b<b^*$ so that 
\[\liminf_{b\nearrow b^*}v(b)\geq \liminf_{b\nearrow b^*}g(b)= g(b^*).\]
Furthermore using Corollary \ref{koro:unabh}
\begin{align*}
v(b^*-\epsilon)-g(b^*)&=E_{b^*-\epsilon}(\rho^\tau g(X_\tau))-v(b^*)\\
&\leq E_{b^*-\epsilon}(\rho^\tau g(X_\tau))-E_{b^*}(\rho^{\tau_{b^*+\epsilon}} g(X_{\tau_{b^*+\epsilon}}))\\
&=\sum_{i=1}^m\left(\Phi_i^{b^*}(b^*-\epsilon)E(g(R_i+b^*))-\Phi_i^{b^*+\epsilon}(b^*)E(g(R_i+\epsilon+b^*))\right)\rightarrow 0
\end{align*}
as $\epsilon\searrow0$. This proves $\limsup_{b\nearrow b^*}v(b)\geq g(b^*)$.
\end{proof}

\begin{figure}[ht]
\begin{center}
\includegraphics[width=8.5cm]{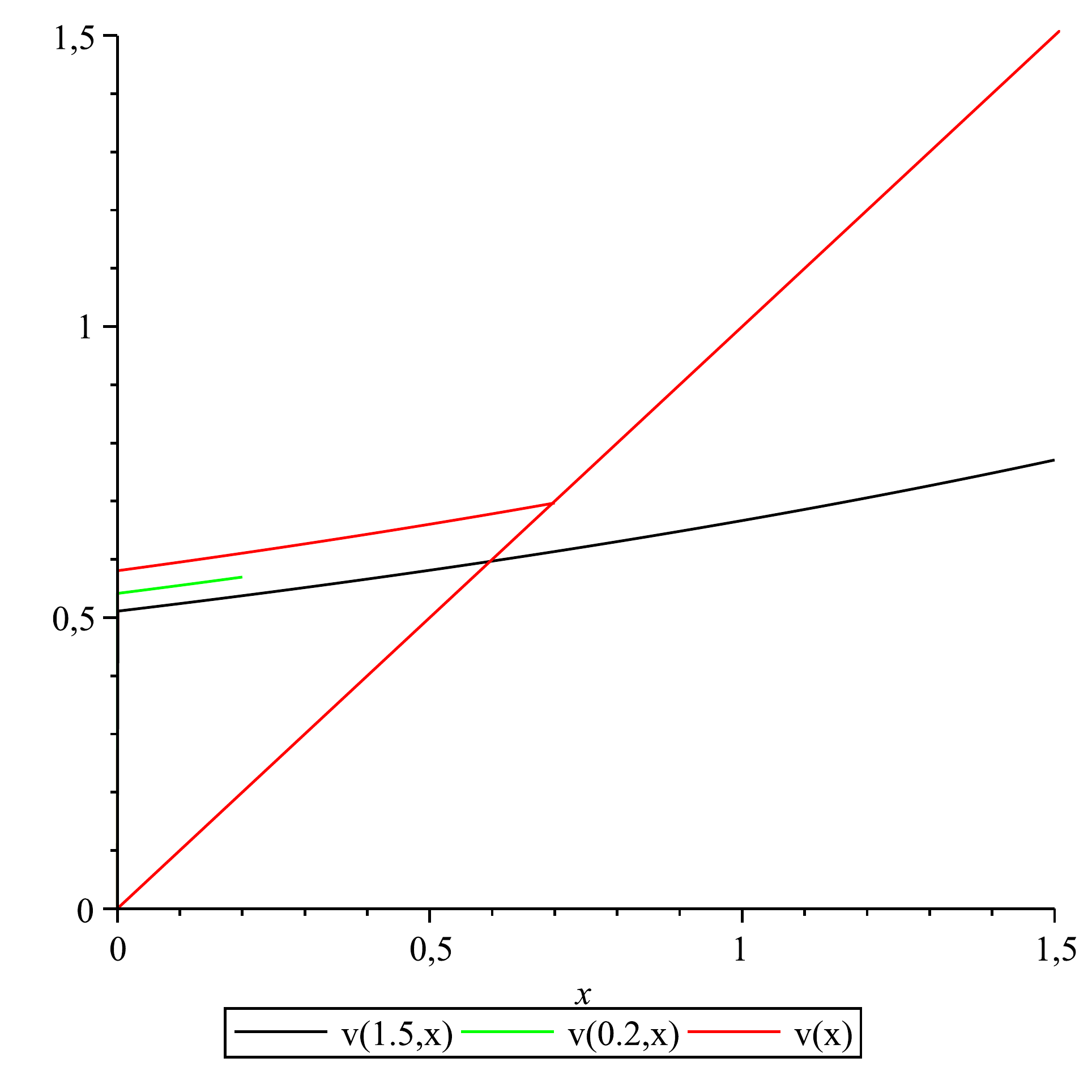}
\caption{Some candidate solutions for different thresholds in the case $g(x)=x$}\label{fig:cont_fit_auto}
\end{center}
\end{figure}

Figure \ref{fig:cont_fit_auto} illustrates how the continuous fit principle can be used: We consider the candidate solutions
\[v(b,x)=\begin{cases}
  \Psi_x(b)  &,~ x< b \\
  g(x) &,~ x\geq b
\end{cases}\]
and solve the equation $\Psi_{b-}(b)=g(b)$, where $\Psi_\cdot$ is defined as in the previous section. If the equation has a unique solution we can conclude that this solution must be the optimal threshold as illustrated in the following Section.

\section{Example}\label{example_autoreg}
We consider the gain function $g(x)=x$ and $\Exp(\mu)$-distributed innovations; in this setting we always assume $(X_n)_{n\in\N_0}$ to have values in $[0,\infty)$. The discussion in Section \ref{threshold} guarantees that the optimal stopping time is of threshold-type. The optimal threshold can be found by the continuous fit principle described in the previous section:\\
The problem is solved if we find a unique $b^*\in[0,\infty)$ that solves the equation
\[b=\Psi_{b-}(b)=\Phi_{b-}(b)(b+\frac{1}{\mu})=\frac{\rho\sum_{k\in\N_0}(\rho,\lambda)_k \frac{(\mu b \lambda)^k}{k!}}{\sum_{k\in\N_0}(\rho,\lambda)_k\frac{(\mu b)^k}{k!}}(b+\frac{1}{\mu}),\]
where we used Theorem \ref{satz:exp} in the last step. This equation is equivalent to
\begin{align*}
&\sum_{k=0}^\infty(\rho,\lambda)_k\frac{\mu^k}{k!}b^{k+1}=\sum_{k=0}^\infty\rho(\rho,\lambda)_k\frac{\mu^k\lambda^k}{k!}b^{k+1}+\sum_{k=0}^\infty\frac{\rho}{\mu}(\rho,\lambda)_k\frac{\mu^k\lambda^k}{k!}b^{k}\\
\mbox{i.e.~~}& \frac{\rho}{\mu}-\sum_{k=0}^\infty(\rho,\lambda)_k\frac{\mu^k}{k!}(1-\rho\lambda^k-\frac{\rho}{\mu}(1-\rho\lambda^k)\frac{\mu\lambda^{k+1}}{k+1})b^{k+1}=0\\
\mbox{i.e.~~}& f(b)=0,
\end{align*}
where 
\[f(b)=\frac{\rho}{\mu}-\sum_{k=0}^\infty(\rho,\lambda)_{k+1}\frac{\mu^k}{k!}(1-\frac{\rho\lambda^{k+1}}{k+1})b^{k+1}.\]
Note that $f(0)=\frac{\rho}{\mu}>0$ and
\[f'(b)=-\sum_{k=0}^\infty(\rho,\lambda)_{k+1}\frac{\mu^k}{k!}(1-\frac{\rho\lambda^{k+1}}{k+1})(k+1)b^{k}<0 \mbox{~for all~}b\in[0,\infty).\]
Since $f(b)\leq \frac{\rho}{\mu}-(1-\rho)(1-\rho\lambda)b$ we furthermore obtain $f(b)\rightarrow-\infty$ for $b\rightarrow\infty$. Hence there exists a unique solution $b^*$ of the transcendental equation $f(b)=0$.\\
The optimal stopping time is
\[\tau^*=\inf\{n\in\N:X_n\geq b^*\}\]
and the value function is given by 
\[v(x)=\begin{cases}
  (b+\frac{1}{\mu})\rho\cdot\frac{\sum_{k\in\N_0}(\rho,\lambda)_k \frac{(\mu x \lambda)^k}{k!}}{\sum_{k\in\N_0}(\rho,\lambda)_k\frac{(\mu b^*)^k}{k!}}  &,~ x< b^* \\
  x &,~ x\geq b^*.
\end{cases}\]
In Figure \ref{fig:cont_fit_auto} $v$ is plotted for the parameters $\mu=1$, $\rho=\lambda=1/2$.\\

\bibliographystyle{plain}
\bibliography{ar1phase}

\begin{thebibliography}{10}

\bibitem{As}
S{\o}ren Asmussen.
\newblock {\em Applied probability and queues}, volume~51 of {\em Applications
  of Mathematics (New York)}.
\newblock Springer-Verlag, New York, second edition, 2003.
\newblock Stochastic Modelling and Applied Probability.

\bibitem{CI}
S{\"o}ren Christensen and Albrecht Irle.
\newblock A note on pasting conditions for the {A}merican perpetual optimal
  stopping problem.
\newblock {\em Statist. Probab. Lett.}, 79(3):349--353, 2009.

\bibitem{CIN}
S{\"o}ren Christensen, Albrecht Irle, and Alexander~A. Novikov.
\newblock {An Elementary Approach to Optimal Stopping Problems for AR(1)
  Sequences}.
\newblock Submitted to Sequential analysis, 2010.

\bibitem{F}
William Feller.
\newblock {\em An introduction to probability theory and its applications.
  {V}ol. {II}}.
\newblock John Wiley \& Sons Inc., New York, 1966.

\bibitem{F2}
Mark Finster.
\newblock Optimal stopping on autoregressive schemes.
\newblock {\em Ann. Probab.}, 10(3):745--753, 1982.

\bibitem{FS}
Marianne Fris�n and Christian Sonesson.
\newblock Optimal surveillance based on exponentially weighted moving averages.
\newblock {\em Sequential Analysis}, 25(10):379--403, 2006.

\bibitem{gr}
George Gasper and Mizan Rahman.
\newblock {\em Basic hypergeometric series}, volume~96 of {\em Encyclopedia of
  Mathematics and its Applications}.
\newblock Cambridge University Press, Cambridge, second edition, 2004.
\newblock With a foreword by Richard Askey.

\bibitem{gm}
Charles~M. Goldie and Ross~A. Maller.
\newblock Stability of perpetuities.
\newblock {\em Ann. Probab.}, 28(3):1195--1218, 2000.

\bibitem{mi}
Valeria~I. Mikhalevich.
\newblock Bayesian choice between two hypotheses for the mean value of a normal
  process.
\newblock {\em Visnik Kiiv. Univ.}, 1:101--104, 1956.

\bibitem{N}
Alexander~A. Novikov.
\newblock {On Distributions of First Passage Times and Optimal Stopping of
  AR(1) Sequences}.
\newblock {\em Theory of Probability and its Applications}, 53(3):419--429,
  2009.

\bibitem{NK}
Alexander~A. Novikov and Nino Kordzakhia.
\newblock Martingales and first passage times of {$\rm AR(1)$} sequences.
\newblock {\em Stochastics}, 80(2-3):197--210, 2008.

\bibitem{NS}
Alexander~A. Novikov and Albert~N. Shiryaev.
\newblock On an effective case of the solution of the optimal stopping problem
  for random walks.
\newblock {\em Teor. Veroyatn. Primen.}, 49(2):373--382, 2004.

\bibitem{NS2}
Alexander~A. Novikov and Albert~N. Shiryaev.
\newblock On solution of the optimal stopping problem for processes with
  independent increments.
\newblock {\em Stochastics}, 79(3-4):393--406, 2007.

\bibitem{ps2}
Goran Pe{\v{s}}kir and Albert~N. Shiryaev.
\newblock Sequential testing problems for {P}oisson processes.
\newblock {\em Ann. Statist.}, 28(3):837--859, 2000.

\bibitem{ps}
Goran Pe{\v{s}}kir and Albert~N. Shiryaev.
\newblock {\em Optimal stopping and free-boundary problems}.
\newblock Lectures in Mathematics ETH Z\"urich. Birkh\"auser Verlag, Basel,
  2006.

\bibitem{RP}
George~C. Runger and Sharad~S. Prabhu.
\newblock A {M}arkov chain model for the multivariate exponentially weighted
  moving averages control chart.
\newblock {\em J. Amer. Statist. Assoc.}, 91(436):1701--1706, 1996.

\end{thebibliography}

\end{document}